\documentclass[12pt]{amsart}
\usepackage{amsmath, amsthm, amscd, amsfonts, amssymb, mathrsfs, graphicx, color}
\usepackage[bookmarksnumbered, colorlinks, plainpages]{hyperref}

\newtheorem{theorem}{Theorem}[section]

\newtheorem{corollary}[theorem]{Corollary}

\newtheorem{lemma}[theorem]{Lemma}
\theoremstyle{definition}

\newtheorem{example}[theorem]{Example}

\newtheorem{notation}[theorem]{Notation}

\newtheorem{remark}[theorem]{Remark}

\newtheorem{question}[theorem]{Question}
\begin{document}
	\title[Recognition by the set of exponents]{Recognition by the set of exponents in the prime factorization of the product of element orders}
\author[M. Baniasad Azad \& M. Arabtash]{Morteza Baniasad Azad \& Mostafa Arabtash}
	\address{Abolfazl Street, 22 Bahman Blvd., Bardsir, Kerman 78416-64979, Iran \newline }
	\email{baniasad84@gmail.com}
	\address{School of Information and Engineering, Dalarna University, Sweden\newline }
	\email{mob@du.se}
	
	\thanks{}
	\subjclass[2010]{20D60.}
	
	\keywords{Finite groups, simple group, element orders, product of element orders.}
	
	\begin{abstract}
Let $G$ be a finite group. Let  $\rho(G) =  \prod_{g \in G} o(g)={p_1}^{\alpha_1}  {p_2}^{\alpha_2} \cdots  {p_k}^{\alpha_k}$, where $p_1, p_2, \cdots, p_k$ are distinct prime numbers and $o(g)$ denotes the order of $g \in G$. The set of exponents in the prime factorization of the product of element orders is denoted by $ {\operatorname{Exp}}_{\rho}(G)$, i.e., $ {\operatorname{Exp}}_{\rho}(G)=\{\alpha_1,\alpha_2, \cdots,\alpha_k\}$.
		 
In this paper, we give a new characterization for some groups by $ {\operatorname{Exp}}_{\rho}(G)$. We prove that the groups ${\rm PSL}(2, 5) \times \mathbb{Z}_p$, ${\rm PSL}(2, 7)$ and ${\rm PSL}(2, 11)$ are uniquely determined by $ {\operatorname{Exp}}_{\rho}(G)$. Furthermore, we prove that the groups ${\rm PSL}(2, 5)$ and ${\rm PSL}(2, 13)$ are uniquely determined by the  parameters $ {\operatorname{Exp}}_{\rho}(G)$ and $|G|$.
Additionally, we prove that 
if  ${\operatorname{Exp}}_{\rho}(G) =  {\operatorname{Exp}}_{\rho}(\mathbb{Z}_{2qr})$, then $G \cong {\rm PSL}(2, 5)$  or $G \cong \mathbb{Z}_{2qr}$,
	where  $q$ and $r$ are distinct  odd prime numbers. 
	\end{abstract}

	\maketitle
	\section{\bf Introduction and Preliminary Results}
	All groups considered in the paper are finite.
	The order of $x \in G$ is denoted by $o(x)$ and the product of element orders of $G$ is denoted by
	$\rho(G)$, i.e., $\rho(G)=\prod_{x \in G} o(x)$.
	The cyclic group of order $n$ is denoted by $\mathbb{Z}_n$.
In \cite{PatassinizbMATH06708963},  Garonzi et al. proved that $\rho(G)<\rho(\mathbb{Z}_n)$, when $G$ is a  non-cyclic group of order $n$.
In \cite[Theorem A]{noce}, Domenico et al. proved that if $G$ is
a non-cyclic supersolvable group of order $n$ 
 and also either $G$ is nilpotent or $G$ is not metacyclic, then
 $\rho(G) \leq q^{-n(q-1)/q} \rho(\mathbb{Z}_n)$,
where $q$ is the smallest prime divisor of $n$.

Let $l(G)=\sqrt[|G|]{\rho(G)}/|G|$.
In \cite{Grazian}, Grazian et al. proved that if 
$l(G) > l(D_{2p})$, where $p$ is an odd prime dividing the order of $G$, then  $G$ is $p$-nilpotent.
In \cite{baniAust}, the first author and Khosravi proved that  
if $l(G)>l(S_3)$,  $l(G)>l(A_4)$ or $l(G)>l({\rm PSL}(2,5))$, then
$G$ is nilpotent, supersolvable or   solvable, respectively.
Also, in \cite{baniMal}, they proved that
the groups ${\rm PSL}(2,7)$ and ${\rm PSL}(2,11)$
are uniquely
determined by the product of their element orders.
In \cite{banihamedapplication,banihamedcom}, the first author 
 and colleagues introduced several groups that are uniquely determined by the product of their element orders; for instance,
the groups ${\rm PSL}(2,5) \times \mathbb{Z}_p$, where $p > 5$ is a prime number.

\begin{notation}
	Let  $n={p_1}^{\alpha_1}  {p_2}^{\alpha_2} \cdots  {p_k}^{\alpha_k}$,
where $p_1, p_2, \cdots, p_k$
	are distinct prime numbers.
We set $[n]_{p_i}=\alpha_i$,
$ {\operatorname{Exp}}(n)=\{\alpha_1, \alpha_2,  \cdots, \alpha_k\}=\{[n]_{p_1}, [n]_{p_2}, \cdots, [n]_{p_k}\} $
and $\pi(n)=\{p_1, p_2, \cdots, p_k\}$.
Here, $ {\operatorname{Exp}}(n)$ 
 denotes the set of exponents in the prime factorization of $n$.
 
For a finite group $G$, we define
$\pi(G):= \pi(|G|)$ and
 ${\operatorname{Exp}}_{\rho}(G):={\operatorname{Exp}}(\rho(G))$.
  Let   $\omega(G)$ 
 denote the set of element orders of $G$, i.e., $\omega(G)=\{o(x)|x \in G\}$.
\end{notation}
\begin{example}
	 For example, if $G \cong {\rm PSL}(2,5)$, then  $\rho(G) = 2^{15}\cdot3^{20}\cdot5^{24}$. Therefore
	 $[\rho(G)]_2=15$,  $[\rho(G)]_3=20$,  $[\rho(G)]_5=24$, $\pi(G)=\{2, 3, 5\}$
	  and $ {\operatorname{Exp}}_{\rho}(G)=\{15, 20, 24\}$.
\end{example}

In this paper, we prove that the groups ${\rm PSL}(2, 5) \times \mathbb{Z}_p$, ${\rm PSL}(2, 7)$ and ${\rm PSL}(2, 11)$ are uniquely determined by
the set of exponents in the prime factorization of their product of element orders, where $p>5$ is a prime number.
 Furthermore, we prove that ${\rm PSL}(2, 5)$ and ${\rm PSL}(2, 13)$
are uniquely determined by their orders and  the set of exponents in the prime factorization of  product of element orders.
Additionally, we prove that 
if  ${\operatorname{Exp}}_{\rho}(G) =  {\operatorname{Exp}}_{\rho}(\mathbb{Z}_{2qr})$, then $G \cong {\rm PSL}(2, 5)$  or $ G \cong \mathbb{Z}_{2qr}$,
where  $q$ and $r$ are distinct  odd prime numbers.
In fact, we generalize the results of \cite{baniMal} and \cite[Theorem 2.3]{banihamedcom}.
For the proof of these results, we need the following lemmas.
	\begin{lemma}\cite{frobenius1895verallgemeinerung}
	Let $G$ be a finite group and $m$ be a positive integer dividing $|G|$. If
	$L_m(G) = \lbrace g \in G|g^m = 1\rbrace$, then $m \mid |L_m(G)|$.
\end{lemma}

The number of elements of order $t$ is denoted by $s_t$.
We see that
\begin{align} \label{yek}
	\rho(G)=\prod_{t \in \omega(G)} t^{s_t},
	\qquad
	|G|=\sum_{t \in \omega(G)}s_t.
\end{align}

	We see
that $s_t = k_t \phi(n)$, where $k_t$ is the number of cyclic subgroups of order $t$. If $m$ is a divisor of $|G|$, then by the above lemma we obtain that
\begin{equation} \label{star}
	\phi(m) \mid s_m, \qquad m \mid \sum_{d \mid m} s_d.
\end{equation}

	\begin{lemma} \cite[Lemma 2.2]{baniMal} \label{karan}
	If $\rho(G)={p_1}^{\alpha_1}{p_2}^{\alpha_2}\cdots{p_k}^{\alpha_k}$, where $k\in \mathbb{N}$, then
	\begin{enumerate}
		\item $\pi(G)=\{{p_1}, {p_2}, \cdots, {p_k}\}$,
		\item $|G| \leq 1+{\alpha_1}+{\alpha_2}+\cdots+{\alpha_k}$, with equality if and only if $G$ is a  group have only elements of prime order.
	\end{enumerate}
\end{lemma}
	\begin{lemma}\cite[Proposition 1.1]{TzbMATH06233336} \label{productdirect}
	Let $G_1, G_2, \dots, G_k$ be finite groups having coprime orders. Then
	\begin{align*}
		\rho(G_1 \times G_2 \times \dots \times G_k)= \prod_{i=1}^{k} {\rho(G_i)}^{n_i},\quad \text{where} \quad 
		n_i =\prod_{\substack{j=1 \\ j\neq i}}^{k}|G_i|, i=1, 2, \dots, k.
	\end{align*}
\end{lemma}
	\begin{lemma} \label{mid1}   \cite[Lemma 2.6]{baniAust}
	Let $ G  $ be a finite group satisfying  $ G = P  \rtimes F$, where $P$ is a cyclic $p$-group
	for some prime $p$, $|F| > 1$ and $(p,|F|) = 1$. Then
	\[\rho(G) = \rho(P)^{|C_{F}(P)|} \rho(F)^{|P|} .\]
\end{lemma}
	\begin{lemma}\cite{hallzbMATH02575957} \label{hall} 
	An integer $n = {p_1}^{\alpha_1} \dots {p_k}^{\alpha_k}$
	is the number of Sylow $p$-subgroups of a
	finite solvable group $G$ if and only if  ${p_i}^{\alpha_i} \equiv 1 \pmod{p}$
	for $i = 1, \dots, k$.
\end{lemma}
\begin{lemma}\cite{huppert} \label{huppert}
	Let $ G ={\rm PSL}(2, q) $ where $ q $ is a $p$-power ($ p $ prime). Then\\
	(1) a Sylow $ p $-subgroup $ P $ of $ G $ is an elementary abelian group of
	order $ q $ and the number of Sylow $ p $-subgroup of $ G $ is $ q + 1 $,\\
	(2) $ G $ contains a cyclic subgroup $ A $ of order $\dfrac{q-1}{2}  $ such that $ N_G(u) $
	is a dihedral group of order $ q - 1 $ for every nontrivial element
	$ u \in A $,\\
	(3) $ G $ contains a cyclic subgroup $ B $ of order $ \dfrac{q+1}{2} $ such that $ N_G(u) $
	is a dihedral group of order $ q + 1 $ for every nontrivial element
	$ u \in B $,\\
	(4) the set $ \{ P^x, A^x, B^x | x \in G \} $ is a partition of $ G $.
\end{lemma}
	\section{\bf The main results}
	For the proof of the main results, we need the following lemmas.
	
	\begin{lemma}\label{111}
		If $|{\operatorname{Exp}}_{\rho}(G)|=m$, then $|\pi(G)|\geqslant m$.
	\end{lemma}
\begin{proof}
	The proof is straightforward.
\end{proof}
	\begin{lemma} \label{bhamed}  
		Let $G$ be a finite group such that $|G|=p^{\alpha}m$ and  $(p,m)=1$, where $p$ is a  prime  number. Then
		$m $ divides $ [\rho(G)]_p$.
	\end{lemma}
\begin{proof}
	We have
	\[
	[\rho(G)]_p = \sum_{p^b \parallel t} \left(b \cdot s_t\right) = \sum_b \left[b \cdot \left(\sum_{p^b \parallel t} s_t\right)\right],
	\]
	where \textbf{$p^b \parallel t$} means that $p^b$ divides $t$ but $p^{b+1}$ does not.
	
	Now, if we define
	\[
	B(k) := \left|\left\{x \in G \mid x^k = 1\right\}\right|,
	\]
	then, by Frobenius' theorem, if $k$ divides $|G|$, then $k$ divides $B(k)$. Therefore,
	\[
	\sum_{p^b \parallel t} s_t = B\left(p^b m\right) - B\left(p^{b-1} m\right)
	\]
	is divisible by $m$ for every $b$, and thus $[\rho(G)]_p$ is divisible by $m$.
\end{proof}
\begin{remark} \label{rem}
Lemma~\ref{bhamed} implies that if $G$ is not a $p$-group, then every prime divisor of $|G|$ divides some element of $\operatorname{Exp}_\rho(G)$. In other words, $\pi(G)$ is contained in the set $E(G)$ consisting of all prime numbers dividing some element of $\operatorname{Exp}_\rho(G)$, i.e.,
\[
\pi(G) \subseteq E(G).
\]
Therefore, if $|E(G)| = \left|\operatorname{Exp}_\rho(G)\right|$, then $\pi(G) = E(G)$. In this case, we can determine all prime divisors of $|G|$ solely from $\operatorname{Exp}_\rho(G)$.

In any case, knowing $\operatorname{Exp}_\rho(G)$ provides excellent control over $\pi(G)$.
\end{remark}
	\begin{lemma} \label{bbhamed}  
		Let $G$ be a finite group such that $p$ divides $| G |$, where $p$ is a  prime  number. Then
		$\phi(p) $ divides  $ [\rho(G)]_p$.
	\end{lemma}
	\begin{proof}	
		We note
		that $s_t = m_t \phi(t)$, where $m_t$ is the number of cyclic subgroups of order $t$.
		We  have
		\begin{align*}
			[\rho(G)]_p=\sum_{ \substack{ t \in \omega(G) \\ p^b \| t}} \left(b \cdot s_t\right)  =\sum_{ \substack{ t \in \omega(G) \\ p^b \| t}}
		\left(b \cdot m_t \cdot \phi(t)\right)
		\end{align*}
	Since $p^b \| t$, it follows that $ t=p^b \cdot s $, where $(p,s)=1$.
		Therefore $\phi(t)= \phi(p^b) \phi(s)$. 
		We know $\phi(p) $ divides $ \phi(p^n)$, for any $n \in \mathbb{N}$.  Thus $\phi(p) \mid s_t$, whenever  $p \mid t$.
		Consequently   $\phi(p) \mid [\rho(G)]_p$.
	\end{proof}
	\begin{lemma} \label{bbbhamed}  
		Let $G$ be a finite group of order $pm$ such that  $(p,m)=1$ and $p$ is a  prime  number. Then
		$p \nmid [\rho(G)]_p$.
	\end{lemma}
	\begin{proof}
		On the contrary, let $p $ divides $ [\rho(G)]_p$.
		By $(\ref{yek})$, we  obtain that
		\begin{align*}
			[\rho(G)]_p=\sum_{ \substack{ t \in \omega(G) \\ p \mid t}}s_t  = |G|-\sum_{ \substack{ t \in \omega(G) \\ p\,  \nmid \,t }}s_t  
			= pm - \sum_{d \mid m} s_d.
		\end{align*}
		Therefore $p $ divides $ \sum_{d \mid m} s_d$. 
		On the other hand, using  $(\ref{star})$, $m$ divides $\sum_{d \mid m} s_d$. 
	Also,	since $(p,m)=1$, it follows that   $pm \mid \sum_{d \mid m} s_d$. Thus $[\rho(G)]_p\leq0$, which is a contradiction. Thus we have  $p \nmid [\rho(G)]_p$.
	\end{proof}
	\begin{corollary}\label{2-part}
		For a finite group $G$,  $[\rho(G)]_p$ is even, where $p$ is an odd prime number and $p \mid |G|$. 
		Also, if $2 $ divides $|G|$, then $[\rho(G)]_2$ is odd.
	\end{corollary}
\begin{proof}
 By Lemma \ref{bhamed} and Lemma \ref{bbhamed}, we obtain  the result.
\end{proof}

	\begin{lemma}\label{mohasebe}
	Let $ G = {\rm PSL}(2, q) $ where $ q $ is a $p$-power ($ p $ prime). Then
		$$\rho({\rm PSL}(2,q))={\rho(\mathbb{Z}_{\frac{q-1}{2}})}^{\frac{q(q+1)}{2}} {\rho(\mathbb{Z}_{\frac{q+1}{2}})}^{\frac{q(q-1)}{2}}  p^{(q+1)(q-1)}.$$
\end{lemma}

\begin{proof}
	By Lemma \ref{huppert}  we get the result.
\end{proof}
	\begin{theorem}
		Let $G$ be a finite group. Then
		$G \cong {\rm PSL}(2,11)$  if and only if 
		 $ {\operatorname{Exp}}_{\rho}(G) ={\operatorname{Exp}}_{\rho}({\rm PSL}(2,11))$.
	\end{theorem}
	\begin{proof}
	By Lemma \ref{mohasebe}, we obtain that
	\begin{align*}
	\rho({\rm PSL}(2,11))={\rho(\mathbb{Z}_{5})}^{66} {\rho(\mathbb{Z}_{6})}^{55}  11^{120}=2^{165}\cdot3^{220}\cdot5^{264}\cdot11^{120}.
	\end{align*}
Therefore $$ {\operatorname{Exp}}_{\rho}(G) =\{120, 165, 220, 264\}=\{2^3 \cdot  3\cdot 5, 3\cdot5\cdot11, 2^2\cdot5\cdot11, 2^3\cdot3\cdot11\}.$$
		Using Lemma  \ref{111}, $|\pi(G)|\geqslant 4$.
		Let $\pi(G)=\{{p_1}, {p_2}, \cdots, {p_k}\}$,
		where $p_1<p_2<\cdots<p_k$ and $4\leq k \in \mathbb{N}$.
We have
 $$ {\operatorname{Exp}}_{\rho}(G)=\{120, 165, 220, 264\}=\{2^3 \cdot 3 \cdot 5, 3 \cdot 5 \cdot 11,  2^2 \cdot 5 \cdot 11, 2^3 \cdot 3 \cdot 11\},$$
 and $E(G)=\{2, 3, 5, 11\}$.
By Remark \ref{rem}, $E(G)=\pi(G)$ and so
		$p_1=2$, $p_2=3$, $p_3=5$ and $p_4=11$.
			Using Lemma \ref{karan},
		\begin{align}\label{222}
			2 \cdot 3 \cdot 5 \cdot 11 \leq |G| \leq 1 + 120 + 165 + 220 + 264=770.
		\end{align}
		Therefore $|G|=2 \cdot 3 \cdot 5 \cdot 11$ or $2^2 \cdot 3 \cdot 5 \cdot 11$. 
 By Lemma \ref{bbbhamed},
		we have $[\rho(G)]_2=165$, $[\rho(G)]_3=220$, $[\rho(G)]_5=264$ and $[\rho(G)]_{11}=120$ and so
		$\rho(G) = 2^{165}\cdot3^{220}\cdot5^{264}\cdot11^{120}$.
		By \cite[Theorem 5]{baniMal}, we conclude  that $G \cong {\rm PSL}(2,11)$.
	\end{proof}
	\begin{theorem}
	Let $G$ be a finite group. Then
		$G \cong {\rm PSL}(2,5)$ or $G \cong \mathbb{Z}_{30}$  if and only if  ${\operatorname{Exp}}_{\rho}(G) = {\operatorname{Exp}}_{\rho}({\rm PSL}(2,5))$.
	\end{theorem}
	\begin{proof}
			By Lemma \ref{mohasebe}, we obtain 
		\begin{align*}
			\rho({\rm PSL}(2,5))={\rho(\mathbb{Z}_{2})}^{15} {\rho(\mathbb{Z}_{3})}^{10}  5^{24}=2^{15}\cdot3^{20}\cdot5^{24}.
		\end{align*}
	Using Lemma \ref{productdirect}, we see that
			\begin{align*}
		\rho(\mathbb{Z}_{30})={\rho(\mathbb{Z}_{2})}^{15} {\rho(\mathbb{Z}_{3})}^{10} {\rho(\mathbb{Z}_{5})}^{6}=2^{15}\cdot3^{20}\cdot5^{24}.
	\end{align*}
		Therefore $ {\operatorname{Exp}}_{\rho}(G) =\{15,20,24\}=\{3 \cdot 5,2^2 \cdot 5, 2^3 \cdot3\}$ and so $E(G)=\{2, 3, 5\}$.
		By Lemma  \ref{111}, we have $|\pi(G)|\geqslant 3$.
By Remark \ref{rem}, $\pi(G)=E(G)$.
		Using Lemma \ref{karan},
		$$	2 \cdot 3 \cdot 5  \mid |G| \leq 1 + 15 + 20 + 24=60.$$
		Therefore $|G|=2 \cdot 3 \cdot 5$ or $2^2 \cdot 3 \cdot 5$.
 By Lemma \ref{bbbhamed},
		we have $[\rho(G)]_2=15$, $[\rho(G)]_3=20$ and $[\rho(G)]_5=24$  and so
		$\rho(G) = 2^{15}\cdot3^{20}\cdot5^{24}$.
		By \cite[Theorem 3]{baniMal}, we obtain that $G \cong {\rm PSL}(2,5)$ or $G \cong \mathbb{Z}_{30}$.
	\end{proof}
	\begin{theorem}
		Let $G$ be a finite group. Then
		$G \cong {\rm PSL}(2,7)$  if and only if  ${\operatorname{Exp}}_{\rho}(G) = {\operatorname{Exp}}_{\rho}({\rm PSL}(2,7))$.
	\end{theorem}
	\begin{proof}
			Using Lemma \ref{mohasebe}, we have
		\begin{align*}
			\rho({\rm PSL}(2,7))={\rho(\mathbb{Z}_{3})}^{28} {\rho(\mathbb{Z}_{4})}^{21}  7^{56}=2^{105}\cdot3^{56}\cdot7^{48}.
		\end{align*}
		Therefore $ {\operatorname{Exp}}_{\rho}(G) = \{48, 56, 105\}=\{2^4\cdot3,   2^3\cdot7, 3\cdot5\cdot7  \}$
		and so $E(G)=\{2, 3, 5, 7\}$.
		By Lemma  \ref{111}, $|\pi(G)|\geqslant 3$.
Using Remark \ref{rem}, we have $\pi(G) \subseteq  E(G)=\{2, 3, 5, 7\}$. 
		Now we consider the following cases separately.
		\begin{itemize}
			\item Let $ \pi(G) = \{2, 3, 5\} $. If $p=3$, then by lemma \ref{bhamed}, $ 10 \mid [\rho(G)]_3 \in \{48, 56,105\}$, which is a contradiction.
			\item Let $ \pi(G) = \{2, 3, 7\} $. 
			Using Lemma \ref{karan},
			$	2 \cdot 3 \cdot 7  \mid |G| \leq 1 + 48 + 56 + 105=210.	$
			If $7^{\alpha}\mid |G|$, where $\alpha\geq2$, then $294=2 \cdot 3 \cdot 7^2 \leq 2 \cdot 3 \cdot 7^\alpha \leq210$, which is a contradiction.
			Therefore $[|G|]_7=1$. By Lemma \ref{bbbhamed},
			we have $[\rho(G)]_2=105$, $[\rho(G)]_7=48$. Thus  $[\rho(G)]_3=56$  and so
			$\rho(G) = 2^{105}\cdot3^{56}\cdot 7^{48}$.
			By  \cite[Theorem 4]{baniMal}, we see that $G \cong {\rm PSL}(2, 7)$.
			\item Let $ \pi(G) = \{2, 5, 7\} $. 
			If $p=7$, then by lemma \ref{bhamed}, $ 10 \mid [\rho(G)]_7 \in \{48, 56,105\}$, which is a contradiction.
			\item Let $ \pi(G) = \{2, 3, 5, 7\} $.
			If $p=7$, then by lemma \ref{bhamed}, $ 30 \mid [\rho(G)]_7 \in \{48, 56,105\}$, which is a contradiction.
		\end{itemize}
	The proof is now complete.
	\end{proof}

			\begin{theorem}
			Let $G$ be a finite group. Then
			$G \cong {\rm PSL}(2,13)$ or $G \cong  {\rm SmallGroup}(546, 13)=\mathbb{Z}_{14} \times (\mathbb{Z}_{13} \rtimes \mathbb{Z}_{3})$  if and only if  ${\operatorname{Exp}}_{\rho}(G) = {\operatorname{Exp}}_{\rho}({\rm PSL}(2,13))$.
		\end{theorem}

	\begin{proof}
					Using Lemma \ref{mohasebe}, we have
		\begin{align*}
			\rho({\rm PSL}(2,13))={\rho(\mathbb{Z}_{6})}^{91} {\rho(\mathbb{Z}_{7})}^{78}  13^{168}=2^{273}\cdot3^{364}\cdot7^{468}\cdot13^{168}.
		\end{align*}
		Using Lemma \ref{productdirect} and Lemma \ref{mid1}, we obtain that
	\begin{align*}
		\rho({\mathbb{Z}_{14} \times (\mathbb{Z}_{13} \rtimes \mathbb{Z}_{3})})
		={\rho(\mathbb{Z}_{14})}^{39} {\rho(\mathbb{Z}_{13} \rtimes \mathbb{Z}_{3})}^{14} =
		(\rho(\mathbb{Z}_{2})^7\rho(\mathbb{Z}_{7})^2)^{39}(\rho(\mathbb{Z}_{13}) \rho(\mathbb{Z}_{3})^{13})^{14},
	\end{align*}
Thus $	\rho({\mathbb{Z}_{14} \times (\mathbb{Z}_{13} \rtimes \mathbb{Z}_{3})})=2^{273}\cdot3^{364}\cdot7^{468}\cdot13^{168}$.
		Therefore $$ {\operatorname{Exp}}_{\rho}(G) = \{168, 273, 364, 468\}=\{2^3\cdot 3 \cdot 7,		3 \cdot 7 \cdot 13,		2^2 \cdot 7 \cdot 13, 2^2\cdot3^2\cdot 13\}.$$
		Thus $E(G)=\{2, 3, 7, 13\}$.
		By Lemma  \ref{111}, $|\pi(G)|\geqslant 4$.
Remark \ref{rem} implies that $E(G)=\pi(G)$.
		Using Lemma \ref{karan},
		\begin{align}
			2 \cdot 3 \cdot 7 \cdot 13 \mid |G| \leq 1 + 168 + 273 + 364 + 468=1274.
		\end{align}
		Therefore $[|G|]_3=[|G|]_7=[|G|]_{11}=1$.
		By Lemma \ref{bbbhamed},
		we have $[\rho(G)]_2=273$, $[\rho(G)]_3=364$, $[\rho(G)]_7=468$ and $[\rho(G)]_{13}=168$  and so
		$\rho(G) = 2^{273}\cdot3^{364}\cdot7^{468} \cdot 13^{168}$.
		By \cite[Theorem 6]{baniMal}, we see that $G \cong {\rm PSL}(2,13)$ or
		$G \cong {\rm SmallGroup}(546, 13)=\mathbb{Z}_{14} \times (\mathbb{Z}_{13} \rtimes \mathbb{Z}_{3})$.
	\end{proof}

\begin{theorem}
	Let $p > 5$ be a prime number. Then
	$G \cong {\rm PSL}(2, 5) \times \mathbb{Z}_p$   if and only if  ${\operatorname{Exp}}_{\rho}(G) = {\operatorname{Exp}}_{\rho}({\rm PSL}(2, 5) \times \mathbb{Z}_p)$.
\end{theorem}
\begin{proof}
	Using Lemma \ref{productdirect}, we see that
	$$\rho({\rm PSL}(2, 5) \times \mathbb{Z}_p)=\rho({\rm PSL}(2, 5))^p\rho(\mathbb{Z}_p)^{60}=(2^{15}3^{20}5^{24})^p(p^{p-1})^{60}.$$
	Therefore ${\operatorname{Exp}}_{\rho}(G) = \{15p, 20p, 24p, 60(p-1)\}$.
	
	By Corollary \ref{2-part}, we have $[\rho(G)]_2=15p$. Therefore $|G|=2^{\alpha } m$, where $m$ is an odd number.
Using Lemma \ref{bhamed}, we deduce  $m$	divides  $[\rho(G)]_2=15p$.
On the other hand,
	by Lemma  \ref{111}, $|\pi(G)|\geqslant 4$.
	Therefore $|G|=2^\alpha\cdot 3  \cdot 5 \cdot p$.   
	Applying  Lemma \ref{bbbhamed},
	we have $[\rho(G)]_5=24p$, $[\rho(G)]_p=60(p-1)$, and  $[\rho(G)]_3=20p$. Therefore  
	$\rho(G) = 2^{15p}\cdot3^{20p}\cdot5^{24p} \cdot p^{60(p-1)}$.
	By \cite[Theorem 2.3]{banihamedcom}, we see that $G \cong {\rm PSL}(2,5) \times \mathbb{Z}_p$.
\end{proof}	

\begin{remark}
	We note that
	\begin{align*}
&	{\operatorname{Exp}}_{\rho}(\mathbb{Z}_{60}) = {\operatorname{Exp}}_{\rho}({\rm PSL}(2, 5) \times \mathbb{Z}_2),\\
&	{\operatorname{Exp}}_{\rho}(\mathbb{Z}_{90}) = {\operatorname{Exp}}_{\rho}({\rm PSL}(2, 5) \times \mathbb{Z}_3),\\
&{\operatorname{Exp}}_{\rho}(\mathbb{Z}_{150}) = {\operatorname{Exp}}_{\rho}({\rm PSL}(2, 5) \times \mathbb{Z}_5).
		\end{align*}
Therefore the groups ${\rm PSL}(2, 5) \times \mathbb{Z}_p$, where $p \in\{2,3,5\}$, are not characterizable by the set ${\operatorname{Exp}}_{\rho}(G)$.
\end{remark}

\begin{theorem}
	Let $G$ be a finite group. Let  $q$ and $r$ be distinct  odd prime numbers. Then
		$G \cong {\rm PSL}(2, 5)$  or $G \cong \mathbb{Z}_{2qr}$  if and only if  ${\operatorname{Exp}}_{\rho}(G) =  {\operatorname{Exp}}_{\rho}(\mathbb{Z}_{2qr})$.
\end{theorem}
\begin{proof}
	Using Lemma \ref{productdirect}, we see that
	$$\rho(\mathbb{Z}_{2qr})=\rho(\mathbb{Z}_2)^{qr}\rho(\mathbb{Z}_q)^{2r}\rho(\mathbb{Z}_r)^{2q}=2^{qr}q^{2r(q-1)}r^{2q(r-1)}.$$
	Therefore ${\operatorname{Exp}}_{\rho}(G) = \{{qr}, {2r(q-1)}, {2q(r-1)}\}$.
	
	By Corollary \ref{2-part}, we have $[\rho(G)]_2=qr$. Therefore $|G|=2^{\alpha } m$, where $m$ is an odd number.
	Using Lemma \ref{bhamed}, we derive $m$	divides  $[\rho(G)]_2=qr$.
	On the other hand,
	by Lemma  \ref{111}, $|\pi(G)|\geqslant 3$.
	Therefore $|G|=2^\alpha\cdot q  \cdot r$.   
	Using Lemma \ref{bbbhamed},
	we have $[\rho(G)]_q=2r(q-1)$, $[\rho(G)]_r=2q(r-1)$. Therefore  
	$\rho(G) = 2^{qr}\cdot q^{2r(q-1)}\cdot r^{2q(r-1)}=\rho(\mathbb{Z}_{2qr})$.
	By \cite[Theorem 2.5]{banihamedapplication}, we see that $G \cong {\rm PSL}(2,5)$ or $G \cong \mathbb{Z}_{2qr}$.
\end{proof}

We end with the following questions.
\begin{question}
	What information about a group $G$ can be obtained from ${\operatorname{Exp}}_{\rho}(G)$?
\end{question}
\begin{question}
	Which groups $G$ can be uniquely determined by the set ${\operatorname{Exp}}_{\rho}(G)$?
\end{question}

\end{document}